\documentclass[12pt]{amsart}
\usepackage{amsmath,amsthm,amsfonts,amssymb,eucal}

\renewcommand {\a}{ \alpha }
\renewcommand{\b}{\beta}

\newcommand{\g}{\gamma}

\newcommand{\vark}{\varkappa}

\newcommand{\s}{\sigma}
\renewcommand{\l}{\lambda}

\renewcommand{\t}{\theta}

\newcommand{\R}{ \mathbb R}

\newcommand {\bx}{\mathbf x}

\newcommand {\bxi}{\boldsymbol\xi}

\newcommand{\CK}{\mathcal K}

\newcommand{\plainC}[1]{\textup{{\textsf{C}}}^{#1}}

\newcommand{\plainH}[1]{\textup{{\textsf{H}}}^{#1}}

\newcommand{\plainL}[1]{\textup{{\textsf{L}}}^{#1}}

\newcommand{\1}
{{\,\vrule depth3pt height9pt}{\vrule depth3pt height9pt}
{\vrule depth3pt height9pt}{\vrule depth3pt height9pt}\,}

\DeclareMathOperator {\re} {{Re}}

\DeclareMathOperator{\op}{{Op}}

\newtheorem{thm}{Theorem}[section]
\newtheorem{cor}[thm]{Corollary}
\newtheorem{lem}[thm]{Lemma}
\newtheorem{prop}[thm]{Proposition}
\newtheorem{cond}[thm]{Condition}

\theoremstyle{definition}

\newtheorem{rem}[thm]{Remark}

\numberwithin{equation}{section}

\newcommand{\bee}{\begin{equation}}
\newcommand{\ene}{\end{equation}}
\newcommand{\bees}{\begin{equation*}}
\newcommand{\enes}{\end{equation*}}
\newcommand{\bes}{\begin{split}}
\newcommand{\ens}{\end{split}}

\newcommand{\bet}{\begin{thm}}
\newcommand{\ent}{\end{thm}}
\newcommand{\bel}{\begin{lem}}
\newcommand{\enl}{\end{lem}}
\newcommand{\bec}{\begin{cor}}
\newcommand{\enc}{\end{cor}}
\newcommand{\bep}{\begin{proof}}
\newcommand{\enp}{\end{proof}}
\newcommand{\ber}{\begin{rem}}
\newcommand{\enr}{\end{rem}}

\newcommand{\CF}{\mathcal F}

\begin{document}
\hoffset -4pc

\title
[Extreme eigenvalues ]
{{Extreme eigenvalues of an integral operator}}
\author[A. V. Sobolev]{Alexander V. Sobolev} 
 \address{Department of Mathematics\\ University College London\\
Gower Street\\ London\\ WC1E 6BT UK}
\email{a.sobolev@ucl.ac.uk}

\begin{abstract}
 
We study the family of compact operators 
$B_\a = V A_\a V$, $\a>0$ in $\plainL2(\R^d)$, $d\ge 1$, 
where $A_\a$ is the pseudo-differential operator 
with symbol $a_\a(\bxi) = a(\a\bxi)$, and 
both functions $a$ and $V$ are real-valued  and decay at infinity. 
We assume that $a$ and $V$ attain their maximal values $A_0>0$, $V_0>0$ 
only at $\bxi = 0$ and $\bx = 0$. We also assume that 
\begin{align*}
a(\bxi)  = &\ A_0 - \Psi_\g(\bxi) + o(|\bxi|^\g),\ |\bxi|\to 0, \\
V(\bx) = &\ V_0 - \Phi_\b(\bx) + o(|\bx|^\b),\ |\bx|\to 0,
\end{align*}
with some functions $\Psi_\g(\bxi)>0$, $\bxi\not =0$ 
and $\Phi_\b(\bx) >0$, $\bx\not = 0$ that are homogeneous of 
degree $\g>0$ and $\b >0$ respectively. 
The main result 
is the following asymptotic formula for the eigenvalues 
$\l_\a^{(n)}$ of the operator $B_\a$ (arranged in descending order counting multiplicity) for fixed $n$ and $\a\to 0$:
\begin{equation*}
\l_\a^{(n)} = A_0V_0^2  - \mu^{(n)} \a^\s + o(\a^\s), \a\to 0,
\end{equation*}
where $\s^{-1} = \g^{-1}+ \b^{-1}$, and 
$\mu^{(n)}$ are the eigenvalues (arranged in ascending order 
counting multiplicity) of the 
model operator $T$ with symbol $V_0^2\Psi_\g(\bxi) + 2A_0 V_0 \Phi_\b(\bx)$.
\end{abstract}

\maketitle

 \section{Introduction and main result}\label{main:sect}
 
Let $a=a(\bxi), \bxi\in \R^d$, 
$V = V(\bx), \bx\in \R^d$, $d\ge 1$, 
be bounded real-valued functions 
such that $a(\bxi)\to 0$, $V(\bx)\to 0$ as 
$|\bxi|\to\infty, |\bx|\to\infty$.
Consider the self-adjoint operator on $\plainL2(\R^d)$ defined by 
\begin{equation*}
B_\a = V \CF^* a_\a \CF V,\  a_\a(\bxi) = a(\a \bxi), \a>0,
\end{equation*}
where $\CF$ is the unitary Fourier transform 
\begin{equation*}
(\mathcal F u)(\bxi) = \hat u(\bxi) = \frac{1}{(2\pi)^{\frac{d}{2}}}\int 
e^{-i \bxi\cdot\bx} u(\bx) d\bx.
\end{equation*} 
Here and further on the integral 
without indication of the domain means integration over 
the entire space $\R^d$. The operator $\CF^* a_\a \CF$ is also described 
as a pseudo-differential operator with symbol $a_\a$. 
This description however is not helpful for us as we do not use  calculus of 
pseudo-differential operators. 
It is clear that $B_\a$ is compact for all $\a>0$. 
We are interested in the asymptotics of the extreme top 
eigenvalues of the operator 
$B_\a$ as $\a\to 0$. More precisely, 
denote by $\l_\a^{(1)}, \l_\a^{(2)}, \dots$ the eigenvalues of $B_\a$ arranged in 
descending order counting multiplicity. The associated normalized pair-wise orthogonal eigenfunctions are denoted by 
$\psi_\a^{(1)}, \psi_\a^{(2)}, \dots$. 
We study the asymptotics of $\l_\a^{(n)}$ as 
$\a\to 0$ for a fixed $n$. 
This problem has been addressed in the literature in different contexts 
under different 
conditions on the functions $a$ and $V$. For example, if 
$a$ and $V$ are indicator functions of bounded intervals in $\R$, 
the  
behaviour of the 
eigenvalues was studied by D. Slepian and H.O. Pollak in \cite{SP}. 
For $d\ge 2$ this problem was analyzed by D. Slepian in \cite{Slepian} 
with $a, V$ being indicator functions 
of balls. In both cases (one- and multi-dimensional) 
the eigenvalues $\l_\a^{(n)}$ are exponentially close to $1$ as $\a\to 0$. 

In \cite{Widom_61} 
H. Widom considered the function $V$ which was the 
indicator of an interval $I$, and 
symbol $a = a(\xi),\ \xi \in\R,$  
having one global maximum at $\xi = 0$, and satisfying the condition 
\begin{align}\label{wid61:eq}
a(\xi) = A_0 - \Psi |\xi|^\g + o(|\xi|^\g),\ |\xi|\to 0,
\end{align}
with $A_0 = a(0) = \max a(\xi)>0$, 
and some $\Psi >0$, $\g >0$. It was proved that  
\begin{align}\label{Widom:eq}
\l_\a^{(n)} = A_0 - \a^\g \Psi \mu^{(n)} + o(\a^\g),\ \a\to0,
\end{align}
where $\mu^{(n)}, n = 1, 2, \dots$ are eigenvalues of the 
fractional Dirichlet Laplacian 
$(-\Delta)^{\frac{\g}{2}}$ on $I$, arranged in ascending 
order counting multiplicity. A multi-dimensional analogue of 
this result was obtained by H. Widom in \cite{Widom_63}. 
We omit its formulation 
for the sake of brevity. 
A result of the type \eqref{wid61:eq} also 
holds if $V$ is not assumed to be a simple 
indicator function, 
but attains its (positive) 
maximum on a set of positive measure,  see \cite{Rom}.

For applications 
to transport problems 
(see \cite{BW} and \cite{KNR})   
it is also useful to investigate the case where 
both functions $a$ and $V$ 
have unique power-like maxima. 
This is exactly the case that we study in the present paper. 
The precise conditions on $a$ and $V$ are described below.
By $C, c$ (with or without indices) we denote various positive constants whose 
precise value is of no importance.

\begin{cond}\label{aV:cond}
\begin{enumerate}
\item \label{decay:item}
$a$ and $V$ are real-valued $\plainL\infty$-functions such that 
$a(\bxi)\to 0$ as $|\bxi|\to \infty$, and 
$V(\bx)\to 0$ as $|\bx|\to \infty$.
\item \label{global:item}
The functions 
$a$ and $V$ attain their global maxima only 
at $\bxi = \mathbf 0$ and $\bx = \mathbf 0$ respectively:
\begin{align*}
A_0: = \max a(\bxi)>0,\ V_0:= \max V(\bx)>0.
\end{align*}
The function $V$ satisfies the condition $-V_0+c\le V(\bx)\le V_0$, 
$\bx$ a.e.,  
with a positive constant $c$.  
\item \label{asymp:item} 
Let $\Phi_\b, \Psi_\g\in\plainC\infty(\R^d\setminus\{0\})$ 
be some real-valued functions, homogeneous of degree 
$\b>0$ and $\g>0$ respectively, positive at $\bx\not =0$.
%
The functions $V$ and $a$ satisfy the properties
\begin{equation}\label{V:eq}
V(\bx) = V_0 - \Phi_\b(\bx) + o(|\bx|^\b),\ |\bx|\to 0,
\end{equation}
and 
\begin{equation}\label{a:eq}
a(\bxi) = A_0 - \Psi_\g(\bxi) + o(|\bxi|^\g),\ |\bxi|\to 0.
\end{equation}
\end{enumerate}
\end{cond}

The results are described with the help of the following 
model pseudo-differential operator $T$ defined formally by its symbol
\begin{equation}\label{T:eq}
t(\bx, \bxi) = V_0^2 \Psi_\g(\bxi)   +  2 A_0 V_0 \Phi_\b(\bx).
\end{equation}
The operator $T$ is essentially self-adjoint on
$\plainC\infty_0(\R)$, and has a purely discrete
spectrum (see e.g. \cite[Theorems 26.2, 26.3]{Sch}). 
The same operator can be also 
defined (see \cite[p. 229, Theorem 1]{BS}) as the unique self-adjoint operator associated with the quadratic form
\begin{align}\label{tform:eq}
T[u, v] = V_0^2 \int \Psi_\g(\bxi)\hat u(\bxi)\overline{\hat v(\bxi)} d\bxi 
+ 2 A_0 V_0 \int \Phi_\b(\bx) u(\bx) \overline{v(\bx)} d\bx,
\end{align}
which is closed on $D[T] = \plainH{\frac{\g}{2}}(\R)\cap\plainL2(\R, |x|^{\b})$. 
We use the notation $T[u] = T[u, u]$. 
Recall that in view of the polarization identity, 
the form $T[w], w\in D[T]$, determines $T[u, v]$ for all $u, v\in D[T]$. 
Denote by $\mu^{(n)}>0$, $n = 1, 2, \dots$ 
the eigenvalues of $T$ arranged in ascending order counting multiplicity, 
and by $\phi^{(n)}$ 
-- an orthonormal basis of corresponding normalized eigenfunctions.

Let $\s$ be the number found from the equation 
\begin{equation*}
\frac{1}{\s} = \frac{1}{\b} + \frac{1}{\g}.
\end{equation*}
The next theorem constitutes the main result of the paper. 

\begin{thm}\label{main:thm} 
Suppose that the functions $a$ and $V$ satisfy Condition 
\ref{aV:cond}. Then for any $n = 1, 2, \dots$, the asymptotics hold:  
\begin{equation}\label{main_symm:eq}
\lim_{\a\to 0} \a^{-\s}(A_0 V_0^2 - \l_\a^{(n)}) = \mu^{(n)}.
\end{equation}
\end{thm}
 
Let us make a few remarks. 
 
Note that formally, the asymptotics \eqref{main_symm:eq} imply 
\eqref{Widom:eq} if one takes $d = 1$ and $\b = \infty$.  

Observe also that a model operator of the form \eqref{T:eq} was featured 
in \cite{MitSob} where the norm of a special 
self-adjoint integral operator with properties similar to $B_\a$, was 
 studied. 

One could also examine the case when one or both of the functions $a, V$ 
attain their respective maximum values at several points, and 
have there the asymptotics of the type \eqref{V:eq} and \eqref{a:eq}. The author believes that 
this problem can be tackled by 
standard methods via decoupling distinct maximal points, thereby reducing 
the issue to the case of a single maximum. 

Conceptually, 
the proof of Theorem \ref{main:thm} follows the paper 
\cite{Widom_61}, but the technical details are quite different: for instance, 
the model operator $T$ replaces the fractional Laplacian 
used in \cite{Widom_61}.

\section{Preliminary estimates. Lower bounds for the top eigenvalues}

Throughout the paper 
we assume that Condition \ref{aV:cond} is satisfied. 
Without loss of generality we may assume that $A_0 = V_0 = 1$.

Using the unitary scaling transformation reduce the studied operator to the operator
\begin{equation*}
B_\a = W_\a \op(b_\a) W_\a,
\end{equation*}
where $W_\a, a_\a$ are defined in the following way:
\begin{equation*}
W_\a (\bx) = V\bigl(\a^{\frac{\g}{\g+\b}} \bx\bigr), \ \ 
b_\a(\bxi) = a\bigl(\a^{\frac{\b}{\g+\b}}\bxi\bigr).
\end{equation*}
Note that slightly abusing the notation we use for 
the unitarily equivalent operator the same 
symbol $B_\a$. This will not cause any confusion. 
For thus defined functions $W_\a$ and $b_\a$ the conditions 
\eqref{V:eq} and \eqref{a:eq} imply that 
\begin{align}\label{blim:eq}
\underset{\a\to 0}\lim\ \a^{-\s}\bigl(1-b_\a(\bxi)\bigr)
= \Psi_\g(\bxi),\ \forall\bxi\in\R^d,  
\end{align}
and
\begin{align}\label{Wlim:eq}
\underset{\a\to 0}\lim \ \a^{-\s}\bigl(1-W_\a(\bx)^2\bigr)
= 2\Phi_\b(\bx),\ \forall\bx\in\R^d. 
\end{align}
Both convergences are uniform in $\bx$ and $\bxi$ varying over 
compact sets. 

Here is another useful property of the family $W_\a$:

\begin{lem}\label{Vest:lem}
For any $u\in D[T]$, we have
\begin{equation}\label{Vest:eq}
\a^{-\s}\int |W_\a(\bx) - 1|^2 |u(\bx)|^2 d\bx
\to 0, \ \a\to 0.
\end{equation}
\end{lem}

\begin{proof} 
The function $W_\a-1$ 
is bounded uniformly in $\bx$ and $\a$, so that 
\begin{align*}
|W_\a(\bx)-1|^2\le C |W_\a(\bx)-1|,\ \bx\in\R^d.
\end{align*}
On the other hand, 
\begin{align*}
|W_\a(\bx)-1|\le C\a^\s |\bx|^\b,\ \bx\in\R^d.
\end{align*}
Therefore, for any $R>0$, we can estimate as follows:
\begin{align*}
\a^{-\s}\int |W_\a(\bx) - &\ 1|^2 |u(\bx)|^2 d\bxi\\[0.2cm]
\le &\ \a^{-\s}\int_{|\bx| < R} 
|W_\a(\bx) -1|^2 |u(\bx)|^2 d\bx
+ C\a^{-\s}
\int_{|\bx| \ge R} 
|W_\a(\bx) -1| |u(\bx)|^2 d\bx\\[0.2cm]
\le &\ C\a^{\s} \int_{|\bx|<R} |\bx|^{2\b} | u(\bx)|^2 d\bx
 + C\int_{|\bxi|>R} |\bx|^\b |u(\bx)|^2 d\bx\\[0.2cm]
\le &\ C\a^{\s} R^\b\int_{|\bx|<R} |\bx|^{\b} |u(\bx)|^2 d\bxi 
 + C\int_{|\bx|>R} |\bx|^\b |u(\bx)|^2 d\bx.
\end{align*}
Both integrals on the right-hand side 
are finite, since $u\in D[T]$, and the second one tends 
to zero as $R\to \infty$.  
Thus, passing first to the limit $\a\to 0$, and 
then taking $R\to \infty$, we conclude that the right-hand side tends to zero as $\a\to 0$, as claimed.  
\end{proof}

Now we show that in some suitable sense 
the operator $B_\a$ can be approximated by the operator 
$I-\a^\s T$ as 
$\a\to 0$. 
Define the form
\begin{equation}\label{ralpha:eq}
R_{\a}[u] =  
(B_{\a}u, u) - \|u\|^2 + \a^{\s} T[u],
\end{equation} 
which is closed on the domain $D[T]$, and two more forms 
\begin{align}
K_\a[u, v] 
= &\ \a^{-\s}\int (1-b_\a(\bxi)) \hat u(\bxi) \overline{ \hat v(\bxi)}
d\bxi, \label{Kform:eq}\\[0.2cm]
S_\a[u, v] = &\ \a^{-\s}\int \bigl(1-W_\a(\bx)^2\bigr) 
u(\bx)\overline{v(\bx)}d\bx,\label{Sform:eq}
\end{align}
that are defined for all $u, v\in\plainL2(\R^d)$. 
It is easily checked that with $w_\a = W_\a u, y_\a = W_\a v$, we have 
\begin{align}\label{BKS:eq}
\a^{-\s}\bigl((u, v) - (B_\a u, v)\bigr)
= K_\a[w_\a, y_\a] + S_\a[u, v],\ 
\end{align}
and 
\begin{align}\label{ra:eq}
R_\a[u, v] = \a^\s \bigl(T[u, v] - K_\a[w_\a, y_\a] - S_\a[u, v]\bigr). 
\end{align} 
Note that $K_\a[u]\ge 0$ and $S_\a[u]\ge 0$ for all $\a>0$.  
Also, due to \eqref{blim:eq} and \eqref{Wlim:eq}, for any 
$u\in D[T]$ we have 
\begin{align}\label{KS:eq}
K_\a[u]\le C\int |\bxi|^\g |\hat u(\bxi)|^2d \bxi,\
S_\a[u]\le C\int |\bx|^\b |u(\bx)|^2 d\bx, 
\end{align}
with a constant $C$ independent of $u$. 
Moreover, for any $u\in D[T]$ we also have 
\begin{align}\label{Klim_0:eq}
\underset{\a\to 0}\lim \ K_\a[u]
= \int \Psi_\g(\bxi)|\hat u(\bxi)|^2 d\bxi,
\end{align}
and 
\begin{align}\label{Slim_0:eq}
\underset{\a\to 0}\lim \ 
S_\a[u]
= 2\int \Phi_\b(\bx) |u(\bx)|^2 d\bx, 
\end{align}
by the Dominated Convergence Theorem.

\begin{lem}\label{remainder:lem} 
For any $u\in D[T]$ and $w_\a = W_\a u$, we have  
\begin{align}\label{Kerr:eq}
K_\a[w_\a-u]\to 0,\ \a\to 0.
\end{align}
Also, for any $u, v\in D[T]$ we have 
\begin{align}\label{remainder:eq}
\a^{-\s} |R_\a[u, v]|\to 0, \a\to 0.
\end{align}
\end{lem}
 
 \begin{proof}   
Proof of \eqref{Kerr:eq}. Estimate: 
\begin{align*}
K_\a[w_\a-u] 
\le   
 &\ C \a^{-\s} \int |w_\a(\bx) - u(\bx)|^2 d\bx\\
 = &\ C \a^{-\s} \int \bigl(1-W_\a(\bx)\bigr)^2 |u(\bx)|^2 d\bx.
\end{align*}
Here we used the fact that $0\le 1-b_\a\le C$ with some constant $C$.  
The right-hand side tends to zero by \eqref{Vest:eq}.

It suffices to prove \eqref{remainder:eq} for $u = v$. Consider 
separately the terms in the representation \eqref{ra:eq}.  
Write:
\begin{align*}
K_\a[w_\a]
= K_\a[u] + 2\re K_\a[u, w_\a-u] + K_\a[w_\a-u]. 
\end{align*}
The last term tends to zero by \eqref{Kerr:eq}. 
Now estimate the second term:
\begin{align*}
|K_\a[u, w_\a-u]|^2
\le K_\a[u] K_\a[w_\a-u].
\end{align*}
In view of \eqref{KS:eq}, the first factor is uniformly bounded, 
and the second one 
tends to zero. 
Thus 
\begin{align*}
K_\a[w_\a] - K_\a[u]\to 0,\ \a\to 0.
\end{align*}
Together with \eqref{Klim_0:eq} and \eqref{Slim_0:eq} this implies that 
\begin{align*}
\underset{\a\to 0}\lim\bigl(K_\a[w_\a] + S_\a[u]\bigr)
= T[u],
\end{align*}
see \eqref{tform:eq}. 
Due to \eqref{ra:eq} this implies \eqref{remainder:eq}. 
\end{proof}

The lower bound for the eigenvalues $\l_\a^{(n)}$, i.e. 
the upper bound for the left-hand side 
of \eqref{main_symm:eq}, is rather straightforward. 

\begin{lem} 
For all $n = 1, 2, \dots$, we have  
\begin{equation}\label{limsup:eq}
\limsup_{\a\to\infty}
\a^{-\s}(1 -  \l^{(n)}_\a)\le \mu^{(n)}.
\end{equation}
\end{lem}

\begin{proof}
Let $\CK_n\subset \plainL2(\R^d)$, $n\ge 1$, be the span of the 
eigenfunctions $\phi^{(1)}, \phi^{(2)}, \dots, \phi^{(n)}$, 
so $\dim \CK_n  = n$. 
By the max-min principle 
(see e.g. \cite[p. 212, Theorem 5]{BS}), 
\[
\l^{(n)}_\a\ge \min (B_{\a} u, u), 
\]
where the minimum is taken over all functions $u\in \CK_n$ 
such that $\|u\|=1$.
Thus by definition \eqref{ralpha:eq} 
\begin{equation*}
\l^{(n)}_\a \ge 1 - \a^{\s}\max_{u\in\CK_n, \|u\|=1} T[u] 
-  n \max_{1\le j,k\le n} |R_\a[\phi^{(j)}, \phi^{(k)}]|. 
\end{equation*}
Since $\{\phi^{(j)}\}$ are eigenfunctions of $T$,
\[
\max_{u\in\CK_n, \|u\|=1} T[u]  = \mu^{(n)},
\]
and the required result now follows from 
\eqref{remainder:eq}. 
\end{proof}

Now we can establish the uniform localization of 
the eigenfunctions $\psi_\a^{(n)}$, $n=1 ,2, \dots$.
Denote 
\begin{equation*}
\t_\a^{(n)}(\bx) = W_\a(\bx) \psi_\a^{(n)}(\bx).
\end{equation*}

\begin{lem} \label{trunc:lem}
For all $n = 1, 2, \dots,$ 
the forms $K_\a[\t_\a^{(n)}]$ 
and $S_\a[\psi_\a^{(n)}]$ are bounded uniformly in $\a$:
\begin{align}\label{KSa:eq}
\underset{\a\to 0}
\limsup \bigl(K_\a[\t_\a^{(n)}] + S_\a[\psi_\a^{(n)}]\bigr)\le \mu^{(n)},
\end{align}
and 
\begin{align}\label{with:eq}
\|\t_\a^{(n)} - \psi_\a^{(n)}\|\to 0,\ \a\to 0. 
\end{align}
Moreover, for all $R>0$ we have
\begin{equation}\label{truncxi:eq}
\underset{\a\to 0}\liminf\ 
\|\widehat{\psi_\a^{(n)}}\chi_R\|^2\ge 1 - C\mu^{(n)} R^{-\g}, 
\end{equation}
and
\begin{equation}\label{truncx:eq}
\underset{\a\to 0}\liminf\ \| \psi_\a^{(n)}\chi_R\|^2\ge 1 - C \mu^{(n)}R^{-\b}.
\end{equation}
with some constant $C$, independent of $n$ and $R$. 
\end{lem}

\begin{proof}
We drop the superscript ``$n$" for brevity. 
According to \eqref{BKS:eq},
\begin{align*}
\a^{-\s}(1-\l_\a) = K_\a[\t_\a] + S_{\a}[\psi_\a].
\end{align*} 
Now \eqref{KSa:eq} follows from \eqref{limsup:eq}.
Now write
\begin{align*}
\|\t_\a-\psi_\a\|^2 = \int \bigl(1-W_\a(\bx)\bigr)^2 |\psi_\a(\bx)|^2 d\bx.
\end{align*}
The straightforward estimate 
\begin{align*}
\frac{1}{2}(1-W_\a)^2\le 1-W_\a = \frac{1-W_\a^2}{1+W_\a}\le C(1-W_\a^2),
\end{align*}
by the definition \eqref{Sform:eq}, 
implies that 
\begin{align*}
\|\t_\a-\psi_\a\|^2\le C\a^\s S_\a[\psi_\a],
\end{align*}
which leads to the convergence $\|\t_\a-\psi_\a\|\to 0$, $\a\to 0$,
in view of \eqref{KSa:eq}. 

Proof of \eqref{truncxi:eq}. 
By Condition \ref{aV:cond}(\ref{global:item}),  
the point $\bxi = \mathbf 0$ is the global maximum  
of $b_\a(\bxi)$, so in view of 
\eqref{a:eq}, for all $|\bxi|>R, R>0$ and all sufficiently small $\a$ 
we have 
\begin{equation*}
b_\a(\bxi) = a\bigl( \a^{\frac{\b}{\g+\b}}\bxi\bigr)\le 1 - CR^{\g} \a^{\s},  
\end{equation*}  
 with some constant $C$. Thus $\a^{-\s}(1-b_\a(\bxi))\ge CR^\g$, and hence
 \begin{align*}
 K_\a[\t_\a]\ge C R^\g  
 \underset{|\bxi|>R}\int |\hat\t_\a(\bxi)|^2 d\bxi,
 \end{align*}
so that, by \eqref{KSa:eq}, 
$\|\hat\t_\a (1- \chi_R)\|^2\le C\mu R^{-\g}$. 
Together with \eqref{with:eq} this leads to \eqref{truncxi:eq}. 

Proof of \eqref{truncx:eq} is similar. 
By Condition \ref{aV:cond}(\ref{global:item}) and by \eqref{V:eq}, 
for all $|\bx| >R$, $R>0$, we have $|W_\a(\bx)|^2\le 1 - CR^\b \a^\s$, and hence 
\begin{align*}
  S_\a[\psi_\a]\ge CR^\b  
 \underset{|\bx|>R}\int |\psi_\a(\bx)|^2 d\bx,
 \end{align*}
so that by 
\eqref{KSa:eq} again, 
$\|\psi_\a (1- \chi_R)\|^2\le C\mu R^{-\b}$. 
This leads to \eqref{truncx:eq}. 
\end{proof}

With the help of Lemma \ref{trunc:lem}, 
in the proof of Theorem \ref{main:thm} we 
show that any weakly convergent sequence of 
the eigenfunctions $\psi^{(n)}_\a$ in fact converges in norm. 
For this we rely on the following result: 
  
\begin{prop}\label{weak:prop}(See \cite[Lemma 12]{MitSob})
Let $f_j\in\plainL2(\R^d)$ be a sequence such that $\|f_j\|\le C$ uniformly in
$j = 1,2, \dots$, and  $f_j(\bx) = 0$ for all $|\bx|\ge \rho >0$
and all $j = 1, 2, \dots$. Suppose that $f_j$
converges weakly to $f\in\plainL2(\R^d)$ as $j\to\infty$, 
and that for some constant $A>0$,
and all $R\ge R_0>0$,
\begin{equation}\label{local:eq}
\underset{j\to\infty}
\liminf\| \hat f_j\chi_R\|
\ge A - C R^{-\vark}, \ \vark >0,
\end{equation}
with some constant $C$ independent of $j, R$. Then $\|f\|\ge A$.
\end{prop}

\section{Proof of Theorem \ref{main:thm}}

As before, we assume that $a$ and $V$ 
satisfy Condition \ref{aV:cond}, and that 
$A_0 = V_0 = 1$. 

The next lemma is the last step towards the proof of Theorem \ref{main:thm}.

\begin{lem} \label{norm:lem} 
Suppose that for some sequence $\a_k>0$, convergent to zero as $k\to \infty$, 
the sequence of eigenfunctions 
$\psi^{(n)}_{\a_k}$ converges weakly 
to $\psi^{(n)}$. Then 
\begin{enumerate}
\item 
The sequence $\psi^{(n)}_{\a_k}$ converges to 
$\psi^{(n)}$ in norm as $k\to\infty$,
\item
The norm limit  $\psi^{(n)}$
belongs to $D[T]$, and  
\begin{equation}\label{quadrat1:eq}
\lim_{k\to\infty}{\a_k}^{-\s}
\bigl( (\psi_{\a_k}^{(n)}, g)
- B_{\a_k}[\psi_{\a_k}^{(n)}, g]\bigr)
=  T[\psi^{(n)}, g],
\end{equation}
for any $g\in D[T]$.
\end{enumerate}
\end{lem}
 
\begin{proof} As before, 
we omit the superscript ``$n$". 
Also for brevity we write $\a$ instead of $\a_k$. 

Proof of (1). 
Due to the formula
\begin{equation*}
\|\psi-\psi_\a\|^2 = 1+\|\psi\|^2 - 2\re (\psi_\a, \psi)\to 1 - \|\psi\|^2,\ 
\a\to 0,
\end{equation*}
it suffices to show that $\|\psi\| = 1$. 

For a number $\rho >0$ 
denote $w_{\a, \rho} = \psi_\a \chi_\rho$, 
$y_{\a, \rho} = \psi_\a (1-\chi_\rho)$.
Thus, by \eqref{truncxi:eq} and \eqref{truncx:eq},  
\begin{align*}
\|\widehat {w_{\a, \rho}}\chi_R\|\ge \|\widehat{\psi_\a}\chi_R\| - 
\|y_{\a, \rho}\| \ge 1 - C\mu R^{-\g} - C(\mu \rho^{-\b})^{\frac{1}{2}}.
\end{align*}
Since $\psi_\a\to \psi$ weakly, then for any $\rho>0$ 
the family $w_{\a, \rho}$  
converges to $\psi\chi_\rho$ weakly. 
Using Proposition \ref{weak:prop} for the sequence 
$w_{\a, \rho}$ we conclude that 
\begin{equation*}
\|\psi \chi_\rho\|\ge 1 - C(\mu\rho^{-\b})^{\frac{1}{2}}.
\end{equation*}
Since $\rho$ is arbitrary, this means that 
$\|\psi\|=1$, which implies the norm convergence 
$\psi_\a\to\psi$, $\a\to 0$, as claimed. 

Proof of (2). 
By Part (1) above, and by \eqref{with:eq}, 
we have 
\begin{equation*}
\|\hat\t_\a - \hat\psi\| \le 
\|\t_\a - \psi_\a\|+ \|\psi_\a - \psi\|\to 0,\ \a\to 0.
\end{equation*} 
Thus for a subsequence $\hat \t_\a$, there is a pointwise convergence 
$\hat\t_\a\to \hat\psi$, $\a\to 0$.  
By \eqref{blim:eq}, the integrand in $K_\a[\t_\a]$ converges pointwise 
to $\Psi_\g(\bxi) |\hat \psi(\bxi)|^2$. 
By \eqref{KSa:eq}, $K_\a[\t_\a]$ 
is uniformly bounded, so by Fatou's Lemma, 
$|\bxi|^{\g/2} \hat\psi\in\plainL2(\R^d)$. 

By \eqref{Wlim:eq}, 
the integrand in $S_\a[\psi_\a]$ converges pointwise to 
$2\Phi_\b(\bx)|\psi(\bx)|^2$. 
By \eqref{KSa:eq}, $S_\a[\psi_\a]$ is uniformly 
bounded, so by Fatou's Lemma again, 
$|\bx|^{\b/2}\psi\in \plainL2(\R^d)$. Together with the previously obtained property 
$|\bxi|^{\g/2} \hat\psi\in\plainL2(\R^d)$, this means that $\psi\in D[T]$.

Proof of \eqref{quadrat1:eq} is similar to that of 
\eqref{remainder:eq}, but is somewhat more complicated since 
it involves functions $\psi_\a$ depending on the 
parameter $\a$. 
By \eqref{BKS:eq},  
\begin{align*}
{\a}^{-\s}
\bigl( (\psi_{\a}, g)
- B_{\a}[\psi_{\a}, g]\bigr) = 
K_\a[\t_\a, y_\a] + S_\a[\psi_\a, g], 
\end{align*}
where $y_\a = W_\a g$. We prove that 
\begin{align}\label{Klim:eq}
\underset{\a\to 0}\lim \ K_\a[\t_\a, y_\a]
= \int \Psi_\g(\bxi) \hat\psi(\bxi) \overline{\hat g(\bxi)} d\bxi,
\end{align}
and 
\begin{align}\label{Slim:eq}
\underset{\a\to 0}\lim \ 
S_\a[\psi_\a, g]
= 2\int \Phi_\b(\bx)\psi(\bx) \overline{g(\bx)} d\bx. 
\end{align}
Estimate:
\begin{align*}
\bigl|K_\a[\t_\a, y_\a] - K_\a[ \t_\a, g]\bigr|^2
\le K_\a[\t_\a]  K_\a[y_\a - g].
\end{align*}
The first factor is bounded uniformly in $\a$
by \eqref{KSa:eq}, and the second one tends to zero 
due to \eqref{Kerr:eq}. 
This shows that 
\begin{equation}\label{K1:eq}
K_\a[\t_\a, y_\a] - K_\a[\t_\a, g]\to 0,\ \a\to 0.
\end{equation}
Because of this property, and 
because of \eqref{KS:eq}, in the proof of \eqref{Klim:eq} 
we may assume that $\hat g$ is compactly supported, i.e. 
$\hat g(\bxi) = 0$ for all $|\bxi| > R$ with some $R>0$.  
The convergence  
\eqref{blim:eq} is uniform in $\bxi: |\bxi|\le R$ for any $R$.
At the same time, as shown earlier, 
$\|\hat \t_\a-\hat\psi\|\to0, \a\to0$, so that 
\begin{equation*}
K_\a[\t_\a, g]\to \int \Psi_\g(\bxi) 
\hat\psi(\bxi)\overline {\hat g(\bxi)} d\bxi,\ \a\to 0.
\end{equation*} 
Together with \eqref{K1:eq} this gives \eqref{Klim:eq}. 

Proof of \eqref{Slim:eq} is simpler. 
Because of \eqref{KS:eq}, 
we may assume that $g$ 
is compactly supported. 
The convergence \eqref{Wlim:eq}
is uniform in $\bx: |\bx|\le R$ for any $R>0$. 
Using the property $\|\psi_\a - \psi\|\to0, \a\to 0,$ 
established in Part 1, we obtain 
\begin{equation*}
S_\a[\psi_\a, g] \to \int 2\Phi_\b(\bx) 
\psi(\bx) \overline {g(\bx)} d\bx,\ \a\to 0,
\end{equation*}
so that \eqref{Slim:eq} holds. 

Put together \eqref{Klim:eq} and \eqref{Slim:eq} to conclude that 
\begin{equation*}
 \a^{-\s}\bigl((\psi_{\a}, g)
- B_\a[\psi_{\a}, g]\bigr) \to T[\psi, g],\ \a\to 0,  
\end{equation*}
as required. 
\end{proof}

\begin{proof}[Proof of Theorem \ref{main:thm}]
 The proof essentially follows the plan of \cite{Widom_61}.
It suffices to show that for
any sequence $\a_k\to 0, k\to \infty, $ one can find a subsequence
$\a_{k_l}\to 0$, $l\to \infty$, such that
\begin{equation}\label{two:eq}
\lim_{l\to \infty} \a_{k_l}^{-\s}(1-\l^{(n)}_{\a_{k_l}}) = \mu^{(n)}. 
\end{equation}
Since $\|\psi_{\a_k}^{(n)}\|=1$, one can extract a 
subsequence $\a_{k_l}\to 0$ such that 
 $\psi_{\a_{k_l}}^{(n)}$ converges weekly as $l\to \infty$. 
By Lemma \ref{norm:lem} 
$\psi^{(n)}_{\a_{k_l}}$ converges in norm as $l\to \infty$.
Denote by $\psi^{(n)}$ its limit, so $\|\psi^{(n)}\| = 1$. 
Further 
for simplicity we write $\psi^{(n)}_\a$ and $\l^{(n)}_\a$ 
instead of $\psi^{(n)}_{\a_{k_l}}$ and $\l^{(n)}_{\a_{k_l}}$.
As $\psi^{(n)}_\a$, $n = 1, 2, \dots$, are pair-wise orthogonal, 
so are their limits $\psi^{(n)}$, $n = 1, 2, \dots$.

Fix a number $n = 1, 2, \dots$. 
For an arbitrary function $f\in D[T]$ write
\begin{equation*}
\a^{-\s}(1-\l^{(n)}_\a)(\psi^{(n)}_\a, f)
= \a^{-\s}\bigl((\psi^{(n)}_\a, f) - B_\a[\psi^{(n)}_\a, f]\bigr).
\end{equation*}
Suppose that $f$ is such that $(\psi^{(n)}, f)\not = 0$. Then, in view of 
\eqref{quadrat1:eq},  
 \begin{equation*}
\lim_{\a\to 0} \a^{-\s}(1-\l^{(n)}_\a) 
= \frac{T[\psi^{(n)}, f]}
{(\psi^{(n)}, f)}.
 \end{equation*}
Let $f = \phi^{(j)}$, where $\phi^{(j)}$ is chosen in such a way that 
$(\phi^{(j)}, \psi^{(n)})\not = 0$. 
This is possible due to the completeness of 
the family $\phi^{(k)}, k = 1, 2, \dots$.   Thus 
 \begin{equation*}
\lim_{\a\to 0} \a^{-\s}(1-\l^{(n)}_\a) = \mu^{(j)}.
 \end{equation*}
 By the
uniqueness of the above limit,
$(\psi^{(j)}, \phi^{(s)}) = 0$ 
for all $s$'s such that $\mu^{(s)}\not = \mu^{(j)}$.
Thus, by completeness of the system $\{\phi^{(k)}\}$,
the function $\psi^{(n)}$  is an eigenfunction 
of $T$ with the eigenvalue $\mu^{(j)}$, i.e. 
$T[\psi^{(n)}] = \mu^{(j)}$. 
As $\psi^{(k)}_\a$, $k = 0, 1, \dots, n$, are pair-wise orthogonal, 
so are their limits $\psi^{(k)}$, $k = 0, 1, \dots, n$. 

Further proof is by induction. 
Let $n=1$, so that by \eqref{limsup:eq}, $\mu^{(j)}\le \mu^{(1)}$, 
and hence $j=1$, and $\psi^{(1)}$ is the eigenfunction of $T$ 
with eigenvalue $\mu^{(1)}$. 
%
Suppose that for some $n$, the 
collection $\psi^{(1)}, \psi^{(2)},\ \dots, \psi^{(n-1)}$ 
are eigenfunctions of $T$ with eigenvalues 
$\mu^{(1)}, \mu^{(2)}, \dots, \mu^{(n-1)}$. 
Since $\psi^{(n)}$ is orthogonal to each 
$\psi^{(k)}$, $k = 1, 2, \dots, n-1$, 
by the standard min-max (or, more precisely, max-min) principle for 
operators semi-bounded from below,  
we have $T[\psi^{(n)}]\ge \mu^{(n)}$, which means that 
$\mu^{(j)}\ge \mu^{(n)}$.  
On the other hand, by \eqref{limsup:eq},
\begin{equation*}
\lim_{\a\to 0} \a^{-\s}(1-\l^{(n)}_\a)\le \mu^{(n)},
\end{equation*}
and hence $\mu^{(j)}\le \mu^{(n)}$. 
Therefore $\mu^{(j)} = \mu^{(n)}$, and $\psi^{(n)}$ is the eigenfunction 
of $T$ with eigenvalue $\mu^{(n)}$. By induction, 
the formula \eqref{two:eq} is proved for all $n$, 
which entails \eqref{main_symm:eq}, 
and hence proves Theorem \ref{main:thm}. 
 \end{proof}

\textbf{Acknowledgements.} 
The author is grateful to R. Romanov for bringing the 
problem studied in this paper to his attention, 
for pointing out references \cite{BW} and \cite{KNR}, and for 
many stimulating discussions.  

The author was supported by EPSRC grant EP/J016829/1.




\end{document}